\documentclass[10pt,leqno]{article}
\usepackage{fullpage}
\usepackage{amssymb}
\usepackage{amsmath}
\usepackage{amsthm}
\usepackage{longtable}
\newtheorem{theorem}{Theorem}

\newtheorem{lemma}{Lemma}
\newtheorem{remark}{Remark}
\begin{document}

\title{\Large  On some properties of representation functions related to the\\
Erd\H{o}s-Tur\'{a}n conjecture}
\author{\large Csaba S\'andor$^{1}$\footnote{
Email:~csandor@math.bme.hu. This author was supported by the OTKA
Grant No. K109789. This paper was supported by the J\'anos Bolyai
Research Scholarship of the Hungarian Academy of Sciences.}
~~Quan-Hui Yang$^{2}$\footnote{Email:~yangquanhui01@163.com. This
author was supported by the National Natural Science Foundation
for Youth of China, Grant No. 11501299, the Natural Science
Foundation of Jiangsu Province, Grant Nos. BK20150889,~15KJB110014
and the Startup Foundation for Introducing Talent of NUIST, Grant
No. 2014r029.} }
\date{} \maketitle
 \vskip -3cm
\begin{center}
\vskip -1cm { \small 1. Institute of Mathematics, Budapest
University of Technology and Economics, H-1529 B.O. Box, Hungary}
 \end{center}

 \begin{center}
{ \small 2. School of Mathematics and Statistics, Nanjing University of Information \\
Science and Technology, Nanjing 210044, China }
 \end{center}

\begin{abstract} For a set $A\subseteq \mathbb{N}$ and $n\in \mathbb{N}$,
let $R_A(n)$ denote the number of ordered pairs $(a,a')\in A\times
A$ such that $a+a'=n$.  The celebrated Erd\H{o}s-Tur\'{a}n
conjecture says that, if $R_A(n)\ge 1$ for all sufficiently large
integers $n$, then the representation function $R_A(n)$ cannot be
bounded. For any positive integer $m$, Ruzsa's number $R_m$ is
defined to be the least positive integer $r$ such that there
exists a set $A\subseteq \mathbb{Z}_m$ with $1\le R_A(n)\le r$ for
all $n\in \mathbb{Z}_m$. In 2008, Chen proved that $R_{m}\le 288$
for all positive integers $m$. Recently the authors proved that
$R_m\ge 6$ for all integers $m\ge 36$. In this paper, for an
abelian group $G$, we prove that if $A\subseteq G$ satisfies
$R_A(g)\le 5$ for all $g\in G$, then $|\{g:g\in G, R_A(g)=0\}|\ge
\frac{1}{4}m-\sqrt{5m}$. This improves a recent result of Li and
Chen. We also give upper bounds of $|\{g:g\in G, R_A(g)=i\}|$ for
$i=2,4$.

{\it 2010 Mathematics Subject Classification:} Primary
11B34,11B13.

{\it Keywords and phrases:}  Representation function, Ruzsa's
number, Erd\H{o}s-Tur\'{a}n conjecture

\end{abstract}

\section{Introduction}

Let $G$ be an abelian group. For any set $A,B\subseteq G$, let
$$R_{A,B}(g)=\sharp \{(a,b):~ a\in A,~b\in B,~a+b=g\}. $$
Let $R_A(g)=R_{A,A}(g)$. If $A\subseteq \mathbb{N}$ and $R_A(n)\ge
1$ for all sufficiently large integers $n$, then we say that $A$
is a basis of $\mathbb{N}$. The celebrated Erd\H{o}s-Tur\'{a}n
conjecture \cite{erdosturan} states that if $A$ is a basis of
$\mathbb{N}$, then $R_A(n)$ cannot be bounded. Erd\H{o}s
\cite{erdos} proved that there exists a basis $A$ and two
constants $c_1,c_2> 0$ such that $c_1 \log n \leq R_A(n) \leq c_2
\log n$ for all sufficiently large integers $n$. Recently,
Dubickas \cite{Dubickas} gave the explicit values of $c_1$ and
$c_2$. In 2003, Nathanson \cite{nathanson} proved that the
Erd\H{o}s-Tur\'{a}n conjecture does not hold on $\mathbb{Z}$. In
fact, he proved that there exists a set $A\subseteq \mathbb{Z}$
such that $1\leq R_A(n) \leq 2$ for all integers $n$. In the same
year, Grekos et al. \cite{Grekos} proved that if $R_A(n)\geq 1$
for all $n$, then $\limsup_{n\rightarrow
 \infty}R_A(n)\geq 6.$ Later, Borwein et al. \cite{Borwein} improved 6 to
 8. In 2013, Konstantoulas \cite{Konstantoulas} proved that if the upper density
 $\overline{d}(\mathbb{N}\setminus (A+A))$ of the set of numbers
 not represented as sums of two numbers of $A$ is less than
 $1/10$, then $R_A(n)>5$ for infinitely many natural numbers $n$.
 Chen \cite{chen12} proved that there exists a basis $A$ of
 $\mathbb{N}$ such that the set of $n$ with $R_A(n)=2$ has density
 one. Later, the second author \cite{yang} and Tang \cite{tang15}
 generalized Chen's result. For the analogue of Erd\H{o}s-Tur\'an conjecture in groups, one can refer
 to \cite{Haddad},~\cite{Haddad08} and \cite{Konyagin}.

For a positive integer $m$, let $\mathbb{Z}_m$ be the set of
residue classes mod $m$. If $R_A(n)\ge 1$ for all $n\in
\mathbb{Z}_m$, then $A$ is called an additive basis of
$\mathbb{Z}_m$.

In 1990, Ruzsa \cite{ruzsa} found a basis $A$ of $\mathbb{N}$ for
which $R_A(n)$ is bounded in the square mean. Ruzsa's method
implies that there exists a constant $C$ such that for any
positive integer $m$, there exists an additive basis $A$ of
$\mathbb{Z}_m$ with $R_A(n)\le C$ for all $n\in \mathbb{Z}_m$. For
each positive integer $m$, Chen \cite{chen08} defined Ruzsa's
number $R_m$ to be the least positive integer $r$ such that there
exists an additive basis $A$ of $\mathbb{Z}_m$ with $R_A(n)\le r$
for all $n\in \mathbb{Z}_m$. In the same paper, Chen proved that
$R_m\le 288$ for all positive integer $m$ and $R_{2p^2}\le 48$ for
all primes $p$.

In 2016, the authors \cite{CsabaYang} proved that if $m\ge 36$,
then $R_m\ge 6$. That is, if $m\ge 36$ and $A\subseteq
\mathbb{Z}_m$ satisfies $R_A(n)\le 5$ for all integers $n$, then
there exists a $n_0\in \mathbb{Z}_m$ such that $R_A(n_0)=0$.
Recently, Li and Chen (see \cite[Corollary 1.3]{chenli}) gave a
quantitative version of this result.

{\noindent \bf Li \& Chen's Theorem}. Let $G$ be a finite abelian
group with $|G|=m$ and $A\subseteq G$. If $R_A(g)\le 5$ for all
$g\in G$, then
$$|\{g:g\in G, R_A(g)=0\}|\ge
\frac{7}{32}m-\frac{1}{2}\sqrt{10m}-1.$$

In this paper, we improve Li and Chen's theorem and also give an
example on the other hand. For convenience, for a fixed
nonnegative integer $i$, we denote the set $\{g:g\in G,
R_A(g)=i\}$ by $S_i$.

\begin{theorem}\label{thm11} (a) Let $G$ be a finite abelian
group with $|G|=m$ and $A\subseteq G$. If $R_A(g)\le 5$ for all
$g\in G$, then $|S_0|\ge \frac{1}{4}m-\sqrt{5m}$.

(b) Let $p$ be a prime and $m=2(p^2+p+1)$. Then there exists a
subset $A\subseteq \mathbb{Z}_m$ such that $R_A(n)\le 5$ for all
$n\in \mathbb{Z}_m$ and $|S_0|<\frac{3}{8}m$.
\end{theorem}

If $R_A(g)\le 5$ for all $g\in G$, then by $|S_0|+|S_2|+|S_4|\le
m$ and Theorem \ref{thm11} (a), we see that $|S_2|+|S_4|\le
\frac{3}{4}m+\sqrt{5m}$. In the next two theorems, we give upper
bounds for $|S_2|$ and $|S_4|$ respectively.

\begin{theorem}\label{thm12} (a) Let $A\subseteq G$ satisfy
$R_A(g)\le 5$ for all $g\in G$. Then $|S_2|\le
\frac{1}{2}m+3\sqrt{5m}$.

(b) Let $p$  be a prime and $m=p^2+p+1$. Then there exists a
subset $A\subseteq \mathbb{Z}_m$ such that $R_A(n)\le 2$ for all
$n\in \mathbb{Z}_m$ and $|S_2|= \frac{1}{2}m-\frac{1}{2}$.
\end{theorem}
\begin{remark} The example in Theorem \ref{thm12} (b) shows that
Theorem \ref{thm12} (a) is nearly best possible.
\end{remark}

If $R_A(g)\le 5$ for all $g\in G$, by the statement before Theorem
\ref{thm12}, we have $|S_2|+|S_4|\le \frac{3}{4}m+\sqrt{5m}$, and
so $|S_4|\le \frac{3}{4}m+O(\sqrt{m})$. It seems difficult to
improve this upper bound. In the following, we will prove this
result by a weak condition $R_A(g)\le 7$ for all $g\in G$.

\begin{theorem}\label{thm13} (a) Let $A\subseteq G$ satisfy
$R_A(g)\le 7$ for all $g\in G$. Then $|S_4|\le
\frac{3}{4}m+O(\sqrt{m})$.

(b) Let $p$  be a prime and $m=2(p^2+p+1)$. Then there exists a
subset $A\subseteq \mathbb{Z}_m$ such that $R_A(n)\le 4$ for all
$n\in \mathbb{Z}_m$ and $|S_4|= \frac{1}{2}m-1$.
\end{theorem}

\section{Preliminary Lemmas}

\begin{lemma}(See \cite[Lemma 3]{CsabaYang}.) \label{lem32} Let $A\subseteq G$ and $c$ be a positive integer.
If $R_A(g)\le c$ for all $g\in G$, then $|A|\le \sqrt{cm}$.
\end{lemma}

\begin{lemma}(See \cite[Singer's Theorem]{Singer}.) \label{lem33} If $l$
is a prime power, then there exists $A\subseteq
\mathbb{Z}_{l^2+l+1}$ such that $R_{A,-A}(n)=1$ for all $n\in
\mathbb{Z}_{l^2+l+1},~n\not=\overline{0}$.
\end{lemma}

\begin{lemma}\label{lem31} If $A$ is a subset of $G$, then for any positive integer $k$ we have
$$\sum_{g\in G}(R_A(g)-k)^2\ge km-(2k-1)|A|+k^2-k.$$
\end{lemma}

\begin{proof} We use Lev and S\' ark\" ozy's argument (see \cite{levsarkozy}) in the following.
\begin{eqnarray*}\sum_{g\in G}(R_A(g)-k)^2&=&\sum_{g\in G}R_A(g)^2-2k\sum_{g\in G}R_A(g)+k^2m\\
&=&\sum_{g\in G}R_{A,-A}(g)^2-2k|A|^2+k^2m\\
&=&\sum_{g\in G\setminus \{0\}}R_{A,-A}(g)^2-(2k-1)|A|^2+k^2m\\
&\ge &
\frac{1}{m-1}\left(\sum_{g\in G\setminus \{0\} }R_{A,-A}(g)\right)^2-(2k-1)|A|^2+k^2m\\
&=&\frac{(|A|^2-|A|)^2}{m-1}-(2k-1)(|A|^2-|A|)-(2k-1)|A|+k^2m\\
&=&(m-1)\left(\bigg(\frac{|A|^2-|A|}{m-1}-\Big(k-\frac{1}{2}\Big)\bigg)^2+k-\frac{1}{4}\right)-(2k-1)|A|+k^2.
\end{eqnarray*}

If $\frac{|A|^2-|A|}{m-1}\ge k$ or $\frac{|A|^2-|A|}{m-1}\le k-1$,
then we have
\begin{eqnarray*}\sum_{g\in G}(R_A(g)-k)^2\ge
k(m-1)-(2k-1)|A|+k^2=km-(2k-1)|A|+k^2-k,
\end{eqnarray*}
and the result is true.

If $k-1<\frac{|A|^2-|A|}{m-1}<k$, then
$$\sum_{g\in G\setminus \{0\}}R_{A,-A}(g)^2\ge
\min_{\substack{k_1,k_2,\ldots,k_{m-1}\in \mathbb{N}\\
\sum_{i=1}^{m-1}k_i=|A|^2-|A|}}\sum_{i=1}^{m-1}k_i^2.$$ It is
known that if $\sum_{i=1}^{m-1}k_i$ is fixed, where $k_i\in
\mathbb{N}$, then $\sum_{i=1}^{m-1}k_i^2$ gets the minimal value
when $|k_i-k_j|\le 1$ for all $1\le i,j\le m-1$. Let
$|A|^2-|A|=q(m-1)+r$, where $q,r$ are nonnegative integers and
$0\le r<m-1$. Then
$q=\left\lfloor\frac{|A|^2-|A|}{m-1}\right\rfloor$ and
$r=\left\{\frac{|A|^2-|A|}{m-1}\right\}(m-1)$. Hence
\begin{eqnarray*}\sum_{g\in G\setminus \{0\}}R_{A,-A}(g)^2&\ge &
\min_{\substack{k_1,k_2,\ldots,k_{m-1}\in \mathbb{N}\\
\sum_{i=1}^{m-1}k_i=|A|^2-|A|}}\sum_{i=1}^{m-1}k_i^2=
rk+(m-1-r)(k-1)\\
&=&\left\{\frac{|A|^2-|A|}{m-1}\right\}(m-1)k^2+\left(1-\left\{\frac{|A|^2-|A|}{m-1}\right\}\right)(m-1)(k-1)^2\\
&=&(k-1)^2(m-1)+(2k-1)\left\{\frac{|A|^2-|A|}{m-1}\right\}(m-1).\end{eqnarray*}
Therefore,
\begin{eqnarray*}\sum_{g\in G}(R_A(g)-k)^2&=&\sum_{g\in G\setminus \{0\}}R_{A,-A}(g)^2-(2k-1)|A|^2+k^2m\\
&\ge
&(k-1)^2(m-1)+(2k-1)\left\{\frac{|A|^2-|A|}{m-1}\right\}(m-1)\\
&&-(2k-1)(|A|^2-|A|)-(2k-1)|A|+k^2m\\
&=&(k-1)^2(m-1)-(2k-1)(m-1)(k-1)-(2k-1)|A|+k^2m\\
&=&km-(2k-1)|A|+k^2-k.
\end{eqnarray*}
\end{proof}

\section{Proofs}

\begin{proof}[Proof of Theorem \ref{thm11}] Let $A$ be a given subset of $G$ such
that $R_A(g)\le 5$ for all $g\in G$. Then
\begin{eqnarray*}\sum_{g\in G}(R_A(g)-3)^2&=&9|S_0|+4|S_1|+|S_2|+|S_4|+4|S_5|\\
&\le &
8|S_0|+3(|S_1|+|S_3|+|S_5|)+(|S_0|+|S_1|+|S_2|+|S_3|+|S_4|+|S_5|).
\end{eqnarray*}
It is clear that $$|S_0|+|S_1|+|S_2|+|S_3|+|S_4|+|S_5|=|\{g:~g\in
G,0\le R_A(g)\le 5\}|=m,$$
$$|S_1|+|S_3|+|S_5|=|\{g:g\in G,2\nmid
R_A(g)\}|=|\{2a:a\in A\}|\le |A|.$$ Hence we have
\begin{eqnarray}\label{eq21}\sum_{g\in G}(R_A(g)-3)^2\le
8|S_0|+3|A|+m.
\end{eqnarray}
On the other hand, by Lemma \ref{lem31}, taking $k=3$, we have
\begin{eqnarray}\label{eq22}\sum_{g\in G}(R_A(g)-3)^2\ge
3m-5|A|+6.\end{eqnarray} Therefore, by \eqref{eq21},\eqref{eq22}
and Lemma \ref{lem32}, it follows that
$$|S_0|\ge \frac{1}{4}m-|A|+\frac{3}{4}\ge
\frac{1}{4}m-\sqrt{5m}.$$

Now we prove part (b). Let $p$ be a prime number and
$m=2(p^2+p+1)$. By Lemma \ref{lem33}, there is a set $B\subseteq
\mathbb{Z}_{p^2+p+1}$ such that $R_{B,-B}(n)=1$ for all $n\in
\mathbb{Z}_{p^2+p+1}$ and $n\not=\overline{0}$. Then for any
integer $l$ with $0\le l\le p^2+p$, we define
$$A_l=2B\cup (2B+2l+1)~\text{mod}~m,\quad
\text{where}~2B=\{2b:~b\in B\}.$$ Now we first prove that
$R_{A_l}(n)\le 4$ for all $n\in \mathbb{Z}_m$.

If $2\mid n$, then $n=a_1+a_2$ with $a_1,a_2\in 2B$ or $a_1,a_2\in
2B+2l+1$. Hence
$R_{A_l}(n)=R_B(\frac{n}{2})+R_B(\frac{n}{2}-(2l+1))\le 2+2=4$.

If $2\nmid n$, then $n=a_1+a_2$ with $a_1\in 2B,a_2\in 2B+2l+1$ or
$a_1\in 2B+2l+1,a_2\in 2B$. Hence,
$R_{A_l}(n)=R_B(\frac{n-2l-1}{2})\times 2\le 4$.

Therefore, $R_{A_l}(n)\le 4$ for all $n\in \mathbb{Z}_m$.

Let $P$ be a statement and we define
$$I(P)=\left\{\begin{array}{l l}
1,& \text{if the statement}~P~\text{is true};\\
0,& \text{if the statement}~P~\text{is false}.\end{array}\right.$$

Let
$$X_{\text{odd}}^l=\{2k+1:~2k+1\in \mathbb{Z}_m~\text{and}~R_{A_l}(2k+1)=0\},$$
$$X_{\text{even}}^l=\{2k:~2k\in \mathbb{Z}_m~\text{and}~R_{A_l}(2k)=0\}.$$

Then $S_0=X_{\text{odd}}^l\cup X_{\text{even}}^l$. It is clear
that $R_{A_l}(2n+1)=0$ if and only if $R_B(n-l)=0$. Then
\begin{eqnarray*}|X_{odd}^l|&=&p^2+p+1-\sharp\{n:n\in
\mathbb{Z}_{p^2+p+1},R_B(n)=2\}-\sharp\{n:n\in
\mathbb{Z}_{p^2+p+1},R_B(n)=1\}\\
&=&p^2+p+1-{p+1\choose
2}-(p+1)=\frac{1}{2}p^2-\frac{1}{2}p<\frac{1}{4}m.\end{eqnarray*}

 and
\begin{eqnarray*}\sum_{l=0}^{p^2+p}|X_{\text{even}}^l|&=&\sum_{l=0}^{p^2+p}|\{n:~n\in
\mathbb{Z}_{p^2+p+1},R_B(n)=0~\text{and}~R_B(n-2l-1)=0\}|\\
&=&\sum_{l=0}^{p^2+p}\sum_{n=0}^{p^2+p}I(R_B(n)=0~\text{and}~R_B(n-2l-1)=0)\\
&=&\sum_{l=0}^{p^2+p}\sum_{n=0}^{p^2+p}I(R_B(n)=0)I(R_B(n-2l-1)=0)\\
&=&\sum_{n=0}^{p^2+p}I(R_B(n)=0)\sum_{l=0}^{p^2+p}I(R_B(n-2l-1)=0)\\
&=&\left(\frac{p^2}{2}-\frac{p}{2}\right)\sum_{n=0}^{p^2+p}I(R_B(n)=0)=\left(\frac{p^2}{2}-\frac{p}{2}\right)^2.
\end{eqnarray*}
Hence there is an integer $l$ such that
$$|X_{even}^l|\le \frac{1}{4}\cdot \frac{(p^2-p)^2}{p^2+p+1}<
\frac{1}{4}(p^2+p+1)=\frac{1}{8}m.$$ Therefore, for this integer
$l$,
$$|S_0|=|X_{odd}^l|+|X_{even}^l|<
\frac{1}{4}m+\frac{1}{8}m=\frac{3}{8}m.$$
\end{proof}

\begin{proof}[Proof of Theorem \ref{thm12}] By Lemma \ref{lem31}, taking $k=2$, we have
\begin{eqnarray}\label{eq41}\sum_{g\in G}(R_A(g)-2)^2\ge 2m-3|A|+2.
\end{eqnarray}
On the other hand,
\begin{eqnarray}\label{eq42}\nonumber\sum_{g\in G}(R_A(g)-2)^2&=&4|S_0|+|S_1|+|S_3|+4|S_4|+9|S_5|\\
&\le& 4(|S_0|+|S_4|)+9(|S_1|+|S_3|+|S_5|)\\
\nonumber&\le & 4(|S_0|+|S_4|)+9|A|.
\end{eqnarray}
Hence, by \eqref{eq41} and \eqref{eq42}, we have $|S_0|+|S_4|\ge
\frac{1}{2}m-3|A|+\frac{1}{2}\ge \frac{1}{2}m-3\sqrt{5m}$. Since
$$|S_0|+|S_1|+|S_2|+|S_3|+|S_4|+|S_5|=m,$$ it follows that
$$|S_2|\le \frac{1}{2}m+3\sqrt{5m}.$$

Now we prove the part (b). By Lemma \ref{lem33}, there exists a
subset $A\subseteq \mathbb{Z}_m$ such that $R_{A,-A}(n)=1$ for all
$n\in \mathbb{Z}_m$, $n\not=\overline{0}$. It is easy to see that
$|A|=p+1$ and $R_A(n)\le 2$ for all $n\in \mathbb{Z}_m$. Hence
$$|S_2|={|A|\choose 2}=\frac{1}{2}(p^2+p+1)-\frac{1}{2}=\frac{1}{2}m-\frac{1}{2}.$$

\end{proof}

\begin{proof}[Proof of Theorem \ref{thm13}] By Lemma \ref{lem31}, taking $k=4$, we have
$$\sum_{g\in G} (R_{A}(g)-4)^2\ge 4m-7|A|+12.$$

On the other hand, by $|S_1|+|S_3|+|S_5|+|S_7|\le |A|$, we have
\begin{eqnarray*}\sum_{g\in G}(R_A(g)-4)^2&=&16|S_0|+9|S_1|+4|S_2|+|S_3|+|S_5|+4|S_6|+9|S_7|\\
&\le &4(|S_0| +|S_2|+|S_4|+|S_6|)+9|A|+12|S_0|-4|S_4|\\
&\le& 4m+9|A|+12|S_0|-4|S_4|.\end{eqnarray*} Hence $|S_4|\le
4|A|+3|S_0|+3$. Since $\sum_{i=0}^{7}|S_i|=m$, it follows that
$$m\ge |S_0|+|S_4|\ge \frac{|S_4|-4|A|-3}{3}+|S_4|=\frac{4}{3}|S_4|-\frac{4}{3}|A|-1.$$
By Lemma \ref{lem32}, we have
$$|S_4|\le \frac{3}{4}m+|A|+\frac{3}{4}\le
\frac{3}{4}m+\sqrt{7m}+\frac{3}{4}.$$

Now we prove the part (b). Let $p$ be a prime and $m=2(p^2+p+1)$.
By Lemma \ref{lem33}, there exists a subset $A_p\subseteq
\mathbb{Z}_{p^2+p+1}$ such that $R_{A_p,-A_p}(n)=1$ for all
$n\not=\overline{0}$. Let $A=2A_p\cup (p^2+p+1+2A_p)\subseteq
\mathbb{Z}_m$.

If $2\mid n$, then
$R_A(n)=R_{A_p}(\frac{n}{2})+R_{A_p}(\frac{n}{2})\le 2+2=4$. If
$2\nmid n$, then $R_A(n)=2R_{A_p}(\frac{n-(p^2+p+1)}{2})\le 4$.
Hence $R_A(n)\le 4$ for all $n\in \mathbb{Z}_m$.
\begin{eqnarray*}|S_4|&=&|\{n:~2\mid n,n\in \mathbb{Z}_m~\text{and}~R_{A_p}(\frac{n}{2})=2\}|\\
&+&|\{n:~2\nmid n,~n\in \mathbb{Z}_m~\text{and}~R_{A_p}(\frac{n-(p^2+p+1)}{2})=2\}|\\
&=&|\{n:n\in
\mathbb{Z}_{p^2+p+1}~\text{and}~R_{A_p}(n)=2\}|+|\{n:~n\in \mathbb{Z}_{p^2+p+1}~\text{and}~R_{A_p}(n-\frac{p^2+p}{2})=2\}|\\
&= &2{p+1\choose 2}=p^2+p=\frac{1}{2}m-1.
\end{eqnarray*}
\end{proof}

\section{Acknowledgement} This work was done during the second
author visiting to Budapest University of Technology and
Economics. He would like to thank Dr. S\'andor Kiss and Dr. Csaba
S\'{a}ndor for their warm hospitality.

\end{document}